\documentclass[11pt]{amsart}
\usepackage{amsmath,amssymb,amsthm}
\usepackage[latin1]{inputenc}
\usepackage{mathrsfs}
\usepackage{version, tabularx, multicol}
\usepackage{graphicx,float,psfrag}
\usepackage{stmaryrd}
\usepackage{color,latexsym,amsfonts}

\newcommand{\reff}[1]{(\ref{#1})}

\theoremstyle{plain}
\newtheorem{theo}{Theorem}[section]
\newtheorem{theo*}{Theorem}
\newtheorem{cor}[theo]{Corollary}

\newtheorem{lem}[theo]{Lemma}
\newtheorem{defi}[theo]{Definition}
\theoremstyle{remark}
\newtheorem{rem}[theo]{Remark}
\newtheorem{ex}[theo]{Example}

\newcommand{\cb}{{\mathcal B}}

\newcommand{\ce}{{\mathcal E}}

\newcommand{\cf}{{\mathcal F}}

\newcommand{\cl}{{\mathcal L}}

\newcommand{\cR}{{\mathcal R}}

\newcommand{\E}{{\mathbb E}}

\newcommand{\N}{{\mathbb N}}
\renewcommand{\P}{{\mathbb P}}

\newcommand{\R}{{\mathbb R}}

\newcommand{\T}{{\mathbb T}}

\newcommand{\bt}{{\mathbf t}}
\newcommand{\bff}{{\mathbf f}}

\newcommand{\ind}{{\bf 1}}

\newcommand{\clo}{{\rm cl}\;}

\newcommand{\Card}{{\rm Card}\;}

\newcommand{\expp}[1]{\mathop {\mathrm{e}^{ #1}}}

\newcommand{\lb}{[\![}
\newcommand{\rb}{]\!]}

\title[Reversal property of the Brownian tree]{Reversal property of the Brownian tree}

\date{\today}

\author{Romain Abraham}
\address{Romain Abraham,
 Laboratoire MAPMO, CNRS, UMR 7349,
 F\'ed\'eration Denis Poisson, FR 2964,
  Universit\'{e} d'Orl\'{e}ns,
 B.P. 6759,
 45067 Orl\'{e}ns cedex 2,
France}
\email{romain.abraham@univ-orleans.fr}

\author{Jean-Fran\c{c}ois Delmas}
\address{Jean-Fran\c{c}ois Delmas,
 Universit\'{e} Paris-Est, CERMICS (ENPC), F-77455, France}
\email{delmas@cermics.enpc.fr}

\begin{document}



\begin{abstract}
  We  consider the Brownian tree introduced by Aldous and the associated
  Q-process which consists in an infinite spine on which are grafted
  independent Brownian trees.  We present a reversal procedure on
  these trees that consists in looking
  at  the tree  downward from  its  top: the  branching points  becoming
  leaves  and  leaves  becoming  branching points.  We  prove  that  the
  distribution of the tree is invariant under this reversal
  procedure, which  provides a better understanding  of previous results
  from Bi and Delmas (2016).
\end{abstract}

\maketitle

\section{Introduction}

Continuous state branching (CB)  processes are stochastic processes that
can  be obtained  as the  scaling limits  of sequences  of Galton-Watson
processes when the initial number of individuals tends to infinity. They
hence  can be  seen as  a model  for a  large branching  population. The
genealogical structure of  a CB process can be described  by a continuum
random tree (CRT) introduced  first by Aldous \cite{a:crtI}  for the quadratic
critical  case, see  also  Le  Gall and  Le  Jan \cite{lglj:bplpep}  and
Duquesne and  Le Gall  \cite{dlg:rtlpsbp} for  the general  critical and
sub-critical  cases,  and Abraham and Delmas \cite{ad:ctv}  for  the  super-critical
case. We shall only consider the  quadratic case; it is characterized by
a branching mechanism $\psi_\theta$:
\[
\psi_\theta(\lambda)=\beta \lambda^ 2+ 2\beta\theta \lambda, \quad
\lambda\in [0, +\infty ),
\]
where $\beta>0$ and $\theta\in \R$. The sub-critical (resp. critical)
case corresponds to $\theta>0$ (resp. 
$\theta=0$). 
The parameter $\beta$ can be seen as a time scaling
parameter, and $\theta$ as a population size parameter. 

In  this  model the  population  dies  out  a.s.   in the  critical  and
sub-critical cases.   In   order  to   model  branching
population with stationary size  distribution, which corresponds to what
is observed  at an  ecological equilibrium, one  can simply  condition a
sub-critical  or  a  critical  CB  to   not  die  out.   This  gives  a
Q-process,  see  Roelly-Coppoleta  and Rouault  \cite{rcr:pdwcfl}  and
Lambert \cite{l:qsdcsbpcne},  which can also  be viewed  as a CB  with a
specific immigration.   The genealogical  structure of the  Q-process in
the stationary regime is a tree with an infinite spine.  This  infinite spine  has to  be removed  if one
adopts  the immigration  point of  view, in  this case  the genealogical
structure  can be  seen  as  a forest  of  trees.   For $\theta>0$,  let
$(Z_t, t\in  \R)$ be this  Q-process in  the stationary regime,  so that
$Z_t$ is  the size of  the population at time  $t\in \R$.  See  Chen and
Delmas \cite{cd:spsmrcatsbp} for studies on this model in a more general
framework. Let $A_t$  be the time to the most  recent common ancestor of
the population living  at time $t$.     According    to     \cite{cd:spsmrcatsbp},    we    have
$\E[Z_t]=1/\theta$,  and $\E[A_t]=3/4\beta\theta$,  so that  $\theta$ is
indeed a population size
parameter and $\beta$ is a time parameter.

\medskip
For $s<t$,  let $M^t_s$  be   the number of individuals  at time $s$
who  have descendants  at time  $t$.   It is  proven in  Bi and  Delmas
\cite{bd:tl}, that for fixed $\theta>0$  a time reversal property holds: in the stationary regime,
the ancestor process $((M^s_{s-r}, r>0), s\in \R)$ is distributed as the
descendant process $((M^{s+r}_s,  r>0), s\in \R)$, see Remark \ref{rem:bd}.  This paper extends  and explains this
identity in law by reversing the  genealogical tree.  The idea is to see
the tree as ranked branches, with each branch being attached to a longer
one (the longest being the infinite spine). Then, re-attach every branch
by its  highest point  on the same  branch. Hence, branching points  become leaves  and leaves
become branching  points.  Call  this operation the  reversal procedure.
Theorem \ref{theo:reversed_CRT} states that, for $\theta\geq 0$, the Brownian CRT
distribution is invariant by the reversal procedure and 
Corollary  \ref{cor:-forest}  states  that,  for  $\theta\geq  0$,  the
distribution  of  the  genealogical  structure  of  the   Q-process  in  the
stationary regime  is also invariant by  the reversal procedure. See a
similar result  in the discrete setting of splitting trees in D\'avila
Felipe and Lambert \cite{dl:trdrf}.

The  paper is  organized  as  follows.  We  first  introduce in  Section
\ref{sec:notations}  the  framework of  real  trees  and we  define  the
Brownian CRT  that describes the  genealogy of  the CB in  the quadratic
case. We define in Section \ref{sec:reverse} the reversal procedure of a
tree and  prove the invariance property  of the Brownian CRT  under this
reversal procedure.   We then extend  the result to the  Brownian forest
that  describes  the  genealogy  of the  stationary  population  in  the
quadratic  (critical and  sub-critical)  case.  

\section{Notations}\label{sec:notations}

\subsection{Real trees}
\label{sec:real}
The study of real trees has been motivated by algebraic and geometric purposes.
See in particular the survey \cite{dmt:tto}. It has been first used in
\cite{epw:rprtrgr} to study random continuum trees, see also
\cite{e:prt}.

\begin{defi}[Real tree]
\label{defi:realtree}
A real tree is a metric space $(\bt,d_\bt)$ such that:
\begin{itemize}
\item[(i)] For every $x,y\in\bt$, there is a unique isometric map $f_{x,y}$
  from $[0,d_\bt(x,y)]$ to $\bt$ such that $f_{x,y}(0)=x$ and $f_{x,y}(d_\bt(x,y))=y$.
\item[(ii)] For every $x,y\in\bt$, if $\phi$ is a continuous injective map from $[0,1]$ to $\bt$ such that $\phi(0) = x$ and
$\phi(1) = y$, then $\phi([0, 1]) = f_{x,y}([0; d_\bt(x,y)])$.
\end{itemize}
\end{defi}

Notice that a real tree is a length space as defined in  \cite{bbi:cmg}. 
We say that a real tree is {\sl rooted } if there is a distinguished
vertex $\partial=\partial_\bt$ which we call the root. Remark that the set $\{\partial\}$ is a rooted tree that only contains the root.

Let $\bt$ be a compact rooted real tree  and let  $x,y\in\bt$. We denote by $\lb x,y\rb$
the   range    of   the   map   $f_{x,y}$    described   in   Definition
\ref{defi:realtree}. We also  set $\lb x,y\lb=\lb x,y\rb\setminus\{y\}$.
We define the out-degree of $x$, denoted by $k_\bt(x)$, as the number of
connected  components of  $\bt\setminus\{x\}$  that do  not contain  the
root. If  $k_\bt(x)=0$, resp. $k_\bt(x)>1$,  then $x$ is called  a leaf,
resp. a branching point. We  denote by $\cl(\bt)$, resp. $\cb(\bt)$, the
set of leaves, resp. of branching points, of $\bt$. A tree is said to be
binary if  the out-degree of  its vertices belongs to  $\{0,1,2\}$.  The
skeleton of the  tree $\bt$ is the  set of points of $\bt$  that are not
leaves:     $\mathrm{sk}(\bt)=\bt\setminus\cl(\bt)$.     Notice     that
$\clo (  \mathrm{sk}(\bt))=\bt$, where  $\clo(A)$ denote the  closure of
$A$.

We denote by $\bt_x$ the
sub-tree of $\bt$ above $x$ i.e.
$$\bt_x=\{y\in\bt,\ x\in \lb\partial,y\rb\}$$
rooted at $x$.
We say that $x$ is an ancestor of $y$, which we denote
by $x\preccurlyeq y$, if $y\in \bt_x$. We write 
$x\prec y$ if furthermore $x\neq y$. Notice that $\preccurlyeq$ is a
partial order on $\bt$. 
We denote by $x\wedge y$ the Most Recent Common Ancestor (MRCA) of $x$ and $y$ in
$\bt$ i.e. the unique vertex of $\bt$ such that
$\lb\partial,x\rb\cap\lb\partial,y\rb=\lb\partial, x\wedge y\rb$.

We denote by
$h_\bt(x)=d_\bt(\partial,x)$ the height of the vertex $x$ in the tree
$\bt$ and by $H(\bt)$ the height of the tree $\bt$:
$$H(\bt)=\max\{h_\bt(x),\ x\in\bt\}.$$
We define the set of  extremal leaves of $\bt$ by:
\[
\cl^*(\bt)=\{y\in \cl(\bt), \, \exists x\in \bt \text{ s.t. }
x\prec y  \text{ and } h_{\bt_x}(y)=H(\bt_x)\}.
\]
In particular, we can have $\cl^*(\bt)\ne \cl(\bt)$, see Example
\ref{ex:L=L}. 


For $\varepsilon >0$, we define the erased tree $r_\varepsilon(\bt)$
(sometimes called in the literature the $\varepsilon$-trimming of the tree $\bt$)
by
\[
  r_\varepsilon(\bt)=\{x\in\bt\backslash\{\partial\},\ H(\bt_x)\ge \varepsilon\}\cup
\{\partial\}.
\]
For $\varepsilon>0$, $r_\varepsilon(\bt)$ is indeed a tree
and $r_\varepsilon(\bt)=\{\partial\}$ for $\varepsilon> H(\bt)$. Notice that
\begin{equation}\label{eq:approximation}
\bigcup_{\varepsilon>0}r_\varepsilon(\bt)=\mathrm{sk}(\bt).
\end{equation}

\begin{lem}
\label{lem:erased_finite}
  For   every compact rooted real tree $\bt$ not reduced to the root,  and   every
  $\varepsilon\in (0, H(\bt))>0$, the erased tree $r_\varepsilon(\bt)$ has
  finitely  many leaves.
\end{lem}

\begin{proof}
  Let $\bt$ be a compact roooted real  tree not reduced to the root, and
  let   $\varepsilon\in (0, H(\bt))$.  We   set   $N$  the   number   of  leaves   of
  $r_\varepsilon(\bt)$.  If   $N=+\infty$,  there  exists   a  (pairwise
  distinct)     sequence     $(y_n,     n\in\N)$    of     leaves     of
  $r_\varepsilon(\bt)$. Then, by definition  the subtrees $\bt_{y_n}$ of
  $\bt$ are pairwise  disjoint and have height  $\varepsilon$. Hence, if
  we choose  for every $n\in\N$ a  point $x_n$ in $\bt_{y_n}$  such that
  $h_{\bt_n}(x_n)=\varepsilon$, the sequence $(x_n, n\in\N)$ satisfies
$$\forall i,j\in\N,\ i\ne j\Longrightarrow d_\bt(x_i,x_j)\ge
2\varepsilon$$
which contradicts the compactness of the tree $\bt$. So $N$ is finite.
\end{proof}

We give a definition of height regularity which implies the uniqueness
of $x^*_\bt$ for all $x\in \bt$. 

\begin{defi}[Height regular]
\label{def:h-reg}
We say that a compact rooted real tree $\bt$ is height-regular if, for every
$\varepsilon>0$, for every $(x,y)\in \cl(r_\varepsilon(\bt))^2\cup
\cb(r_\varepsilon(\bt))^2$,
$$x\ne y\Longrightarrow h_\bt(x)\ne h_\bt(y).$$  
\end{defi}

\begin{lem}
  Let   $\bt$   be   a    compact   height-regular   tree.   For   every
  $x\in  \bt$, there  exists a  unique $x_\bt^*\in\bt_x$  (or
  simply $x^*$ when there is no risk of confusion) such that
$h_{\bt_x}(x_\bt^*)=H(\bt_x)$.
\end{lem}

\begin{proof}
If $x\in\cl(\bt)$, then $\bt_x=\{x\}$ and the lemma holds trivially.

Let $x\in\mathrm{sk}(\bt)$. First, as $\bt_x$ is compact, $H(\bt_x)$ is
finite and there exists at least one point $y\in\bt_x$ such that
$h_{\bt_x}(y)=H(\bt_x)$.

Assume there exists two
distinct points $y,y'\in\bt_x$ such that
$h_{\bt_x}(y)=h_{\bt_x}(y')=H(\bt_x)$.
Then we have $y\wedge y'\in\bt_x$ and $h_{\bt_x}(y \wedge y')< H(\bt_x)$.
We choose $\varepsilon>0$ such that
$\varepsilon < H(\bt_x)-h_{\bt_x}(y\wedge y')$ and we denote by
$y_\varepsilon$ (resp. $y_\varepsilon'$) the unique point in $\lb
y\wedge y',y\rb$ (resp. $\lb
y\wedge y',y'\rb$) such that $d_\bt(y_\varepsilon,y)=\varepsilon$
(resp. $d_\bt(y'_\varepsilon,y')=\varepsilon$). Remark that these points
exist by the particular choice of $\varepsilon$. Then, by definition,
$y_\varepsilon$ and $y_\varepsilon'$ are distinct leaves of
$r_\varepsilon(\bt)$ and have the same height, which contradicts the
fact that $\bt$ is height regular.
\end{proof}

Let $\bt$ be a compact binary height-regular rooted real tree. For $x\in \bt$, the vertex $x^*$ will be called the top
of the tree  $\bt_x$.  For such a tree, we have the equality:
\[
\cl^*(\bt)=\{x^*,\
 x\in \mathrm{sk}(\bt)\}.
\]

By Equation \reff{eq:approximation}, we also have
\[
\cl^*(\bt)=\bigcup_{\varepsilon>0}\{x^*,\ x\in \cl\bigl(r_\varepsilon(\bt)\bigr)\} 
\]
and we deduce from Lemma \ref{lem:erased_finite} that if $\bt$ is height-regular, then $\cl^*(\bt)$ is at most countable.

For every $x\in\bt$, we define the branching point of $x$ on
$\lb\partial,\partial^*\rb$ as
$$\underline{x}=x\wedge \partial^*.$$
For every $y\in\lb\partial,\partial^*\rb$, the sub-tree
(possibly reduced to its root) rooted at $y$ which does not contain
neither $\partial$ nor $\partial^*$ is given by 
$$\tilde\bt_y=\{z\in\bt,\ z\wedge \partial^*=y\}.$$
Notice that $\tilde \bt_y$ is indeed a tree.  Then, for every $x\in\bt$,
we define the maximal height of the subtree $\tilde \bt_{\underline{x}}$
which is  attached on  $\lb\partial,\partial^*\rb$ and  which contains
$x$ by
$$h'_\bt(x)=H(\tilde \bt_{\underline{x}})+h_\bt(\underline{x}).$$
See Figure \ref{fig:bt1} for a simplified picture of $x$, $\underline{x}$, $\bt_x$,
$x^*$, $\tilde \bt_{\underline{x}}$ and $h'_\bt(x)$. 

\begin{figure}[H]
\begin{center}
\psfrag{0}{$\partial$}
\psfrag{0*}{$\partial^*$}
\psfrag{x}{$x$}
\psfrag{x*}{$x^*$}
\psfrag{xb}{$\underline x$}
\psfrag{tx}{$\bt_x$}
\psfrag{ttx}{$\tilde\bt_{\underline x}$}
\psfrag{Ht}{$H(\tilde\bt_{\underline x})$}
\psfrag{ht}{$h'_\bt(x)$}
\includegraphics[width=8cm]{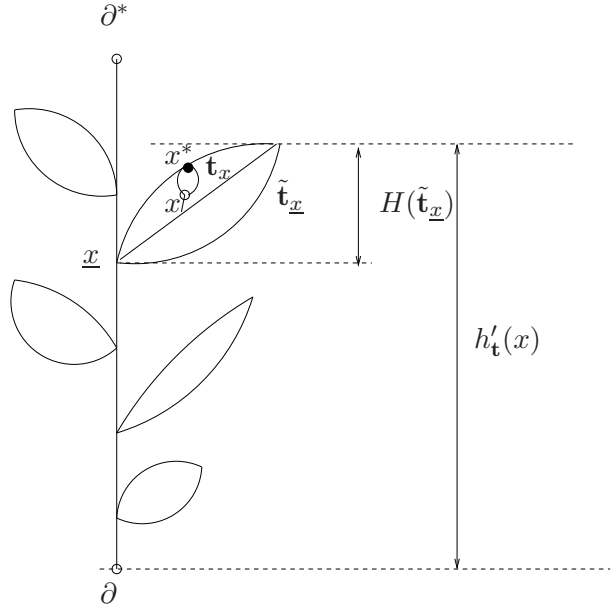}
\caption{A tree $\bt$, with  $x$, $x^*_\bt$ and  $\underline
  x$ elements of $\bt$,   and the sub-trees $\bt_x$ and $\tilde
  \bt_{\underline x}$.}
\label{fig:bt1}
\end{center}
\end{figure}

Let $\bt$ be a compact rooted real tree and let $(\bt_i, {i\in I})$ be a
family of  trees, and $(x_i, {i\in  I})$ a family of  vertices of $\bt$.
We   denote  by   $\bt_i^\circ=\bt_i\setminus\{\partial_{\bt_i}\}$.   We
define  the  tree  $\bt\circledast_{i\in  I}  (\bt_i,x_i)$  obtained  by
grafting the trees $\bt_i$ on the tree $\bt$ at points $x_i$ by
\begin{align*}
 & \bt\circledast_{i\in I}(\bt_i,x_i)=\bt\sqcup\left(\bigsqcup_{i\in I} \bt_i^\circ\right),\\
&
d_{\bt\circledast_{i\in I}
  (\bt_i,x_i)}(y,y') =\begin{cases}
d_\bt(y,y') & \mbox{if }y,y'\in \bt,\\
d_{\bt_i}(y,y') & \mbox{if }y,y'\in \bt_i^\circ,\\
d_\bt(y,x_i)+d_{\bt_i}(\partial_{\bt_i},y') & \mbox{if } y\in\bt\mbox{
  and }y'\in\bt_i^\circ,\\
d_{\bt_i}(y,\partial_{\bt_i})+d_\bt(x_i,x_j)+d_{\bt_j}(\partial_{\bt_j},y')
&  \mbox{if } y\in\bt_i^\circ\mbox{
  and }y'\in\bt_j^\circ\mbox{ with }i\ne j,
\end{cases}\\
 & \partial_{\bt\circledast_{i\in I}(\bt_i,x_i)}=\partial_\bt,
\end{align*}
where $A\sqcup B$ denotes the disjoint union of the sets $A$ and
$B$. Notice that $\bt\circledast_{i\in I}(\bt_i,x_i)$ might not be
compact. 

Let us finish with an instance of a tree $\bt$ such that $\cl^*(\bt)\ne
\cl(\bt)$.

\begin{ex}
      \label{ex:L=L}
For every positive integer $n$, let us set $\bt_n=[0,1/n]\subset \R$, viewed as a rooted real tree when endowed with the usual distance on the real line and rooted at 0. We consider the tree
\[
\bt=\bt_1\circledast _{n\ge 2}(\bt_n,1-\frac{1}{n^2}).
\]
Then $\bt$ is a compact height-regular tree and $1\in \bt_1$ is a leaf of
$\bt$ that does not belong to $\cl^*(\bt)$.
\end{ex}

\subsection{The Gromov-Hausdorff topology}

In order to define random real trees, we endow the set of (isometry classes of) rooted compact real
trees with a metric, the so-called Gromov-Hausdorff metric, which hence
defines a Borel $\sigma$-algebra on this set.

First, let us recall the definition of the Hausdorff distance between
two compact subsets: let $A,B$ be two compact subsets of a metric
space $(X,d_X)$. For every $\varepsilon>0$, we set:
$$A^\varepsilon=\{x\in X,\ d_X(x,A)\le \varepsilon\}.$$
Then, the Hausdorff distance between $A$ and $B$ is defined by:
\[
d_{X, \text{Haus}}(A,B)=\inf\{\varepsilon>0,\  B\subset A^\varepsilon\
\mbox{and}\ A\subset B^\varepsilon\}.
\]

Now, let $(\bt,d_\bt,\partial_\bt)$, $(\bt',d_{\bt'},\partial_{\bt'})$ be two compact
rooted real trees. We define
the pointed Gromov-Hausdorff distance between them, see
\cite{g:msrnrs, epw:rprtrgr},  by:
\[
d_{GH}(\bt,\bt')=\inf\{d_{Z, \text{Haus}}(\varphi(\bt),\varphi'(\bt))\vee
d_Z(\varphi(\partial_\bt),\varphi'(\partial_{\bt'}))\},
\]
where the infimum is taken over all metric spaces $(Z,d_Z)$ and all
isometric embeddings $\varphi:\bt\longrightarrow Z$ and
$\varphi':\bt'\longrightarrow Z$.

Notice that  $d_{GH}$ is only a  pseudo-metric.  We say that  two rooted
real trees $\bt$  and $\bt'$ are equivalent (and  we note $\bt\sim\bt'$)
if there exists a root-preserving  isometry that maps $\bt$ onto $\bt'$,
that  is $d_{GH}(\bt,  \bt')=0$.   This clearly  defines an  equivalence
relation. We  denote by $\T$ the  set of equivalence classes  of compact
rooted real trees.  The Gromov-Hausdorff distance $d_{GH}$ hence induces
a metric  on $\T$ (that  is still  denoted by $d_{GH}$).   Moreover, the
metric   space   $(\T,d_{GH})$   is    complete   and   separable,   see
\cite{epw:rprtrgr}. If $\bt,\bt'$ are two-compact rooted real trees such
that  $\bt\sim   \bt'$,  then,   for  every  $\varepsilon>0$,   we  have
$r_\varepsilon(\bt)\sim r_\varepsilon(\bt')$. Thus, the erasure function
$r_\varepsilon$ is well-defined  on $\T$.  It is easy to  check that the
functions $r_\varepsilon$ for $\varepsilon>0$ are 1-Lipschitz.

Notice that if  $\bt$ is a compact height-regular real  tree, so are all
the trees  equivalent to $\bt$. Let  $\T_0\subset \T$ denote the  set of
equivalence classes of compact binary height-regular real trees.
The next lemma ensures that $\T_0$ is a Borel subset of $\T$. 
\begin{lem}
   \label{lem:T_0}
We have that $\T_0$ is a dense Borel subset of
$\T$. 
\end{lem}

\begin{proof}

Let $\T^f$ (resp. $\T_0^f$) be the subset of trees of $\T$ (resp. $\T_0$) with finitely many leaves. Let $\varepsilon>0$. By Lemma \ref{lem:erased_finite}, we have $r_\varepsilon(\T)\subset \T^f$. Conversely, for every $\bt\in\T^f$, we define
\[
\tilde\bt =\bt\circledast _{x\in\cl(\bt)}([0,\varepsilon],x)
\]
where the segment $[0,\varepsilon]$ is viewed as a rooted real tree when endowed with the usual distance on the real line and with root 0. Then we have $r_\varepsilon(\tilde\bt)=\bt$ and hence $r_\varepsilon(\T)=\T^f$. The same arguments also apply to obtain $\T_0^f=r_\varepsilon(\T_0)$.

Notice that for every $n\ge 1$, the subset $\T^n$ of trees with less than $n$ leaves is a closed subset of $\T$ and that the subset of binary height-regular trees with exactly $n$ leaves is an open set (for the induced topology) of $\T^n$. This implies that $\T^f$ and $\T^f_0$ are Borel sets. 
Then, use  that
$
  \T_0=\bigcap_{\varepsilon>0}r_\varepsilon^{-1}(\T_0^f)
 $
  to get that  $\T_0$ is a
  measurable subset of $\T$.

Using    Definition     \ref{def:h-reg},     it    is easy    to     prove    that
  $\T_0^f$    is dense in   $\T^f$. Since
  $d_{GH}(\bt, r_\varepsilon(\bt))\leq \varepsilon$  for $\bt\in \T$ and
  $\varepsilon>0$, we  deduce that  $\T^f$ is dense in $\T$. This
  implies that  $\T^f_0$, and thus  $\T_0$, is dense  in $\T$.   
\end{proof}

\subsection{Coding a compact real tree by a function and the Brownian
  CRT}\label{sec:CRT}

Let     $\ce$     be     the      set     of     continuous     function
$g:[0,+\infty)\longrightarrow [0,+\infty)$ with compact support and such
that     $g(0)     =     0$.      For     $g\in     \ce$,     we     set
$\sigma(g)=\sup  \{x, \,  g(x)>0\}$.  Let $g\in  \ce$,  and assume  that
$\sigma(g)>0$,  that  is   $g$  is  not  identically   zero.  For  every
$s, t \ge 0$, we set
$$m_g(s, t) = \inf _{r\in [s\wedge t,s\vee  t ]}g(r),$$
and
\begin{equation}\label{eq:dist_g}
d_g(s, t) = g(s) + g(t) - 2m_g(s, t ).
\end{equation}
It is easy to check that  $d_g$ is a pseudo-metric on $[0,+\infty)$.  We
then say that $s$ and $t$ are equivalent  iff $d_g(s, t) = 0$ and we set
$T_g$ the associated quotient space. We  keep the notation $d_g$ for the
induced distance  on $T_g$.   Then the  metric space  $(T_g, d_g)$  is a
compact  real-tree, see  \cite{dlg:pfalt}.   We denote  by  $p_g$  the
canonical  projection  from $[0,  +\infty  )$  to  $T_g$. We  will  view
$(T_g, d_g)$  as a rooted real  tree with root $\partial  = p_g(0)$. We
will call $(T_g,d_g)$ the  real tree coded by $g$, and conversely that $g$ is a contour function of the tree $T_g$. We  denote by $F$ the
application that associates with a function $g\in\ce$ the equivalence class of the tree $T_g$.

Conversely every rooted compact real tree $(T,d)$ can be coded by a
continuous function $g$ (up to a root-preserving isometry), see
\cite{d:ccrtrvf}.

\bigskip
Let           $\theta\in           \R$,          $\beta>0$           and
$B^{(\theta)}=(B^{(\theta)}_t,  t\geq 0)$  be a Brownian motion  with
drift $-2\theta$ and scale $\sqrt{2/\beta}$: for $t\geq 0$,
\[
B^{(\theta)}_t =\sqrt{2/\beta}\, B_t-2\theta t,
\]
where $B$ is a standard Brownian motion. For $\theta\geq 0$, let
$n^{(\theta)}[de]$ denote the It\^o measure on $\ce$ of positive excursions of
$B^{(\theta)}$  normalized such that for $\lambda\geq 0$:
\begin{equation}
   \label{eq:normlalisation}
n^{(\theta)}\left[1- \expp{-\lambda \sigma}\right]=
\psi_\theta^{-1}(\lambda), 
\end{equation}
where $\sigma=\sigma(e)$ denotes the duration (or the length) of the
excursion $e$ and for $\lambda\geq 0$:
\begin{equation}\label{eq:psi}
\psi_\theta(\lambda)=\beta \lambda^2+2\beta\theta\lambda.
\end{equation}
Let $\zeta=\zeta(e)=\max_{s\in [0, \sigma]}(e_s)$ be the maximum of the
excursion. We set $c_\theta(h)=n^{(\theta)}[\zeta\ge h]$ for $h>0$, and we
recall, see  Section 7 in \cite{cd:spsmrcatsbp} for the case $\theta>0$,  that:
\begin{equation}\label{eq:def-c}
c_\theta(h)=
\begin{cases}
   (\beta h)^{-1} &\text{if $\theta=0$}\\
2\theta \, (\expp{2\beta\theta h}-1) ^{-1} &\text{if $\theta>0$.}
\end{cases}
\end{equation}

We  define  the   Brownian  CRT,
$\tau=F (e)$,  as the (equivalence class of the) tree coded by the  positive excursion $e$
under $n^{(\theta)}$. And we define  the measure $\N^{(\theta)}$ on $\T$
as  the ``distribution''  of $\tau$,  that  is the  push-forward of  the
measure $n^{(\theta)}$ by the application $F$. Notice that
$H(\tau)=\zeta(e)$. 

\begin{rem}\label{rem:height_process}
If we translate the former construction into the framework of
\cite{dlg:rtlpsbp}, then, for $\theta\geq 0$,  $B^{(\theta)}$ is the
  height process which codes the Brownian CRT with branching mechanism
  $\psi_\theta$ and it is obtained from the
  underlying L\'evy process $X=(X_t, t\geq 0)$ with $X_t=\sqrt{2\beta}\,
  B_t-2\beta\theta t$. 
\end{rem}

Let     $e$     with     ``distribution''     $n^{(\theta)}(de)$     and
let $(\Lambda_s^a, s\ge 0,  a\ge 0)$ be the local  time of $e$ at  time $s$ and
level $a$.  Then we  define the  local time measure  of $\tau$  at level
$a\ge 0$,  denoted by $\ell_a(dx)$,  as the push-forward of  the measure
$d\Lambda_s^a$   by  the   map  $  F$,  see   Theorem  4.2   in
\cite{dlg:pfalt}. We shall define $\ell_a$ for $a\in \R$ by setting
$\ell_a=0$ for $a\in \R\setminus [0, H(\tau)]$. 

\subsection{Forests}
\label{sec:forest}

A  forest  $\bff$  is  a  family $((h_i,\bt_i), \, i\in  I)$  of  points  of
$\R\times \T$. Using  an immediate extension of  the grafting procedure,
for  an  interval $\mathfrak{I}\subset  \R$,  we  define the  real  tree
$\bff_\mathfrak{I}=\mathfrak{I}\circledast_{i\in I, h_i\in \mathfrak{I}}
(\bt_i,                                                           h_i)$.
For  $\mathfrak{I}=\R$, $\bff_\R$  is an  infinite spine  (the
real line)  on which we  graft the compact  trees $\bt_i$ at  the points
$h_i$ respectively. We shall identify the forest $\bff$ with $\bff_\R$
when the $(h_i, i\in I)$ are pairwise distinct.

Let us denote, for $i\in I$,  by $d_i$ the distance of the tree $\bt_i$
and by $\bt_i^\circ=\bt_i\setminus\{\partial_{\bt_i}\}$ the tree $\bt_i$
without its root. The
distance on $\bff_\mathfrak{I}$ is then defined, for $x,y\in
\bff_\mathfrak{I}$, by: 
\[
d_\bff(x,y)=\begin{cases}
d_i(x,y) & \text{ if }x,y\in\bt_i^\circ,\\
h_{\bt_i}(x)+|h_i-h_j|+h_{\bt_j}(y) & \text{ if } x\in\bt_i^\circ,\ y\in\bt_j^\circ\mbox{ with }i\ne j,\\
|x-h_j|+h_{\bt_j}(y) & \text{ if } x\not\in\bigcup_{i\in I}\bt_i^\circ,\
y\in\bt_j^\circ\\ 
|x-y| & \text{ if } x,y \not\in\bigcup_{i\in I}\bt_i^\circ.
\end{cases}
\]

The next lemma essentially states that $\bff_\R$ is locally compact. See
\cite{adh:nghpdblcmms}    and   the    references   therein    for   the
Gromov-Hausdorff topology on the set of locally compact trees.

\begin{lem}\label{lem:loc-comp}
  Let  $\mathfrak{I}\subset  \R$ be  a  closed  interval. If  for  every
  $a,b\in\mathfrak{I}$, such that $a<b$,  and every $\varepsilon>0$, the
  set $\{i\in I,\ h_i\in[a,b],\  H(\bt_i)>\varepsilon\}$ is finite, then
  the tree $\bff_\mathfrak{I}$ is a complete locally compact real tree. 
\end{lem}

\begin{proof}
  Let $(x_n, \, {n\ge 0})$ be a bounded sequence of $\bff_\mathfrak{I}$.
  If there exists a sub-sequence $(x_{n_k},  \, k \geq 0)$ which belongs
  to $\mathfrak{I}$ (resp.  to $\bt_i^\circ$ for some $i\in I$), then as
  $   \mathfrak{I}$    is a closed interval   (resp.
  $\bt_i^\circ\cup \{h_i\}$ is compact),  this sub-sequence admits at
  least one accumulation point.

  If this  is not the case,  without loss of generality,  we can suppose
  that  $x_n\in\bt_{i_n}^\circ$ with  pairwise  distinct indices  $i_n$.
  Notice  that  the  sequence  $(h_{i_n},   n\geq  0)$  of  elements  of
  $\mathfrak{I}$              is             bounded,              since
  $d_\bff(h_{i_0}, h_{i_n})\leq d_\bff(x_{0}, x_{n})$.  Therefore, as
  $\mathfrak{I}$ is a 
  closed    interval,   there    exists   a    converging   sub-sequence
  $(h_{i_{n_k}}, \, {k\ge 0})$. Let us denote by $h\in \mathfrak{I}$ its
  limit.       Moreover,       using      the       assumption      that
  $\{i\in I,\  h_i\in[a,b],\ H(\bt_i)>\varepsilon\}$ is finite  for all
  $a<b$, we have  $\lim_{n\to+\infty} d_\bff(x_n,h_{i_n})=0$.
  Therefore, the 
  sub-sequence $(x_{n_k}, \, k\geq 0)$ converges to $h$.

In conclusion,
  we get that every bounded sequence of $\bff_\mathfrak{I}$ admits at
  least one accumulation point. This implies that  $\bff_\mathfrak{I}$
  is complete and locally compact.
\end{proof}

We extend the notion of height of a vertex and of the subtree above a vertex for a forest: for $x\in\bff_\R$, either there exists a unique $i\in I$ such that  $x\in\bt_i$ and we set $h_\bff(x)=h_i+h_{\bt_i(x)}$ and $\bt_x$ the subtree above $x$ in $\bt_i$, or $x\in\R$ and we set $h_\bff(x)=x$ and $\bt_x=\{x\}$.

\section{The reversed tree}\label{sec:reverse}

\subsection{Backbones}\label{sec:backbones}

For a compact rooted real tree  $\bt$, we define an increasing family of
backbones        $(B_n(\bt))_{n\in\N}$.        We        denote       by
$S_0(\bt)=\{x\in\bt,\ h_\bt(x)=H(\bt)\}$ the set  of leaves with maximal
height and  we define the  initial backbone as  the set of  ancestors of
$S_0(\bt)$:
$$B_0(\bt)=\bigcup_{x\in S_0(\bt)}\lb\partial,x\rb.$$
Notice that
if the tree $\bt$ is height-regular, then $S_0(\bt)=\{\partial ^*\}$
and $B_0(\bt)=\lb\partial,\partial^*\rb$ is just the spine from the
root of the tree to its top. 

Let $(\tilde \bt^i, i\in I_0)$ be the connected components of
$\bt\setminus B_0(\bt)$. If $\bt^i$ denotes the closure of
$\tilde\bt^i$, we have $\bt^i=\tilde\bt^i\cup\{x_i\}$ for a unique
$x_i\in B_0(\bt)$ which can be viewed as the root of $\bt^i$.
Then, we define the family of backbones recursively: for $n\ge 1$, we
set
$$B_n(\bt)=B_0(\bt)\circledast _{i\in
  I_0}\bigl(B_{n-1}(\bt^i),x_i\bigr).$$

\begin{rem}
We can also use the alternative recursive definition
$$B_n(\bt)=B_{n-1}(\bt)\circledast_{i\in I_{n-1}}(B_0(\hat\bt^i\cup\{y_i\}),y_i),$$
where the family $(\hat\bt^i,i\in I_{n-1})$ is the connected components of
$\bt\setminus B_{n-1}(\bt)$ and $y_i$ is the unique vertex of $\bt$ such
that $\hat\bt^i\cup\{y_i\}$ is closed (and $y_i$ is then considered as the
root of this tree).
\end{rem}

\begin{rem}
It is easy to check that, if $\bt\sim\bt'$ then, for every $n\in\N$,
$B_n(\bt)\sim B_n(\bt')$. So the function $B_n$ is well defined on
$\T$.
\end{rem}

It is easy to check that for $\bt$ a compact rooted real tree,
$\varepsilon>0$:
\begin{equation}
   \label{eq:rB=Br}
r_\varepsilon\circ B_n(\bt)= B_n \circ r_\varepsilon(\bt). 
\end{equation}
By   Lemma   \ref{lem:erased_finite},   we   deduce   that   for   every
$\bt\in  \T$ and  $\varepsilon>0$,
there exists  an integer $N$  (that depends on $\bt$  and $\varepsilon$)
such that
\begin{equation}
\label{eq:r_epsilon}
r_\varepsilon(\bt)=\bigcup_{n=0}^NB_n\bigl(r_\varepsilon(\bt)\bigr)=B_N\circ
r_\varepsilon(\bt).
\end{equation}

\begin{lem}\label{lem:B_n-->t}
Let $\bt$ be a compact rooted real tree not reduced to the root.
\begin{itemize}
\item We have $\clo \left(\bigcup _{n\in \N}B_n(\bt)\right)=\bt$.
\item Furthermore, if $\bt$ is height-regular and binary, then we have $\displaystyle \bigcup_{n\in\N} \cl(B_n(\bt))=\cl^*(\bt)$.
\end{itemize}
\end{lem}

\begin{proof}
Let $\bt$ be a compact rooted real tree not reduced to the root.
Let $x\in \mathrm{sk}(\bt)$ and set $\varepsilon=H(\bt_x)>0$. By
definition $x\in r_\varepsilon(\bt)$, and using \reff{eq:r_epsilon} as
well as the inclusion $B_n(r_\varepsilon(\bt))\subset B_n(\bt)$, we get  $x\in \bigcup _{n\in\N}B_n(\bt)$, which proves that $\mathrm{sk}(\bt)\subset \bigcup _{n\in
  \N}B_n(\bt)$. Then the first point follows from the fact that $\clo (\mathrm{sk}(\bt))=\bt$.

For the  second point, let us  suppose that $\bt$ is  height-regular and
binary,  and  let $x\in  \cl(B_n(\bt))$  for  some $n\in\N$.   Then,  by
definition  of $B_n(\bt)$,  $x$ is  the  top of  a subtree  of the  form
$\bt_y$, with $y\prec  x$ and, as $\bt$ is  height-regular, it therefore
belongs to $\cl^*(\bt)$.  Conversely,  let $x\in\cl^*(\bt)$.  Then there
exists  $y\in   \mathrm{sk}(\bt)$  such   that  $y^*=x$.   Let   us  set
$\varepsilon=d(y,x)>0$.     Then    $y\in    r_\varepsilon(\bt)$    and,
by \reff{eq:r_epsilon},  $y\in  B_n(\bt)$  for  some  $n\in\N$.   And  by
definition, we have $x=y^*\in\cl(B_n(\bt))$ for the same $n$.
\end{proof}

\subsection{Reversed tree}

The reversal of a tree is only defined for a height-regular binary tree $\bt$. As
already noticed, since $\bt$ is height regular, we have
$S_0(\bt)=\{\partial^*\}$ and
$B_0(\bt)=\lb\partial,\partial^*\rb$. Similarly, using the notations of Section
\ref{sec:backbones}, for every $i\in I_0$, as $\bt^i$ is also
height-regular, we have $B_0(\bt^i)=\lb x_i,x_i^*\rb$. For every $i\in
I_0$, we set $y_i'$
the unique point of $B_0(\bt)$ which is at the same height as $x_i^*$:
$$y_i'\in \lb \partial,\partial^*\rb,\ h_\bt(y'_i)=h_\bt(x_i^*).$$

We then  define recursively the  reversed backbones as follows.  We set,
for $n\ge 0$,
\[
\cR_0(\bt)=(\lb\partial^*,\partial\rb,d,\partial^*).
\]
(notice that the root of $\cR_0(\bt)$ is $\partial^*$) and for $n\ge 1$,
\[
\cR_n(\bt)=\cR_0(\bt)\circledast _{i\in
  I_0}\bigl(\cR_{n-1}(\bt^i),y_i'\bigr).
\]

The reversal procedure is illustrated on Figure \ref{fig:reversed}, the dashed
lines show where the trees are grafted on the reversed tree. Notice
that, for aesthetic purpose, inside a sub-tree, the branches are drawn from left to right
in decreasing order of their height.
\begin{figure}[H]
\begin{center}
\includegraphics[width=8cm]{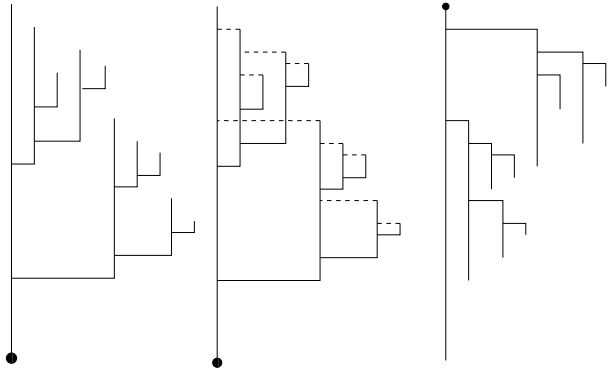}
\caption{A  backbone  $B_3(\bt)$  on  the left  and  its  reversed  tree
  $\cR_3(\bt)$ on the  right. The root of each tree  is represented by a
  bullet.}\label{fig:reversed}
\end{center}
\end{figure}

Intuitively, the leaves of $\cR_n(\bt)$
correspond to branching points of $B_n(\bt)$ (or to its root) and
conversely. Therefore, it is easy to check that $\cR_n(\bt)\in\T_0$
for every $n\in\N$. 

\begin{lem}
   \label{lem:Rnre}
We have for $\bt\in \T_0$:
\begin{equation}
   \label{eq:RB=R}
\cR_n\circ B_n (\bt)=\cR_n(\bt) , \quad
\cR_n\circ \cR_n (\bt)=B_n(\bt) 
\quad\text{and}\quad 
\cR_n\circ r_\varepsilon(\bt)= r_\varepsilon\circ \cR_n(\bt).
\end{equation}
\end{lem}

\begin{proof}
The first two  equalities are obvious and the last one is also obvious if
$\bt$ has a finite number of leaves. We just check that the last
equality holds for general $\bt\in \T_0$. We have:
\[
\cR_n\circ r_\varepsilon(\bt)=
\cR_n\circ B_n\circ r_\varepsilon(\bt)=
\cR_n\circ r_\varepsilon\circ B_n(\bt)=
r_\varepsilon\circ \cR_n\circ B_n(\bt)=
r_\varepsilon\circ \cR_n(\bt),
\]
   where we use the first equality of \reff{eq:RB=R} for the first
   equality, \reff{eq:rB=Br} for the second, the last equality of
   \reff{eq:RB=R} which holds  for $B_n(\bt)$ as it is height-regular and has a finite number of
 leaves for the third and \reff{eq:RB=R} for the last. 
\end{proof}

Furthermore, the sequence of trees $(\cR_n(\bt), {n\ge
  0})$ is non-decreasing. We endow $\bigcup_{n\ge
  0}\cR_n(\bt)$ with the natural distance denoted by $d^\cR$
and we define the reversed tree $\cR(\bt)$ as the completion of
$\bigcup_{n\ge
  0}\cR_n(\bt)$ with respect to the distance $d^\cR$. We give some
properties of the map $\cR$. 

\begin{cor}
   \label{cor:CR}
The map $\cR$ is a one-to-one measurable involution
(that is $\cR\circ
\cR (\bt)=\bt$) defined on $\T_0$. 
\end{cor}

\begin{proof}
  Since $\cR_n(\bt)$ belongs to $\T_0$ for  all $n\in \N$, we deduce that
  $\cR(\bt)$ belongs  to $\T_0$.  The second  equality of \reff{eq:RB=R}
  and   Lemma  \ref{lem:B_n-->t}   readily  imply   that  $\cR$   is  an
  involution. It is therefore one-to-one.

Recall that the set  $r_\varepsilon(\T_0)$ is the set $\T^f_0$ of compact, height-regular trees with a finite
number of leaves. It is easy to see that $\cR$ is continuous when restricted to
$r_\varepsilon(\T_0)$. This gives that $\cR\circ r_\varepsilon$ is
measurable. Use the third equality in \reff{eq:RB=R} to deduce that 
$r_\varepsilon\circ \cR$ is measurable. As, for every $\bt\in\T$, we
have $d_{GH}(\bt,r_\varepsilon(\bt))\le \varepsilon$, we get that 
$\cR=\lim_{\varepsilon\to 0}r_\varepsilon\circ\cR$, 
which implies that $\cR$ is measurable. 
\end{proof}

\begin{rem}
There is no natural extension of $\cR$ to $\T$ (in particular because
$\cR$ is not uniformly continuous of $\T_0$). 
\end{rem}

\subsection{Reversed CRT}

We first check that  the Brownian CRT is height-regular. 

\begin{lem}\label{lem:reversed_CRT}
Let $\theta\ge 0$. Let $\tau$ be a Brownian CRT under the excursion measure $\N^{(\theta)}$. Then, we have that
$\N^{(\theta)}$-a.e., $\tau\in\T_0$. 
\end{lem}

\begin{proof}
Let $h>0$. Following \cite{np:rpetpeodbm,np:bpbe}, we say that a
process $X$ admits a $h$-minimum (resp. a $h$-maximum) at time $t$
if there exist $s<t$ and $u>t$ such that $X_s=X_u=X_t+h$
(resp. $X_s=X_u=X_t-h$) and $X_r\ge X_t$ (resp. $X_r\le X_t$) for
every $r\in[s,u]$.

Then, if we denote by $e$ an  excursion under $n^{(\theta)}$ and $\tau$
the associated real tree, for a.e. $h$ the branching points
of $r_h(\tau)$  correspond to the $h$-minima of $e$ and each
leaf of $r_h(\tau)$ is associated with an $h$-maxima of
$e$. As $n^{(\theta)}$-a.e., two local extrema of the excursion $e$ have
different levels, we get  that $\tau\in\T_0$, $\N^{(\theta)}$-a.e. by
definition of $\T_0$. 
\end{proof}

Let $\tau$ be a Brownian CRT under the excursion measure
$\N^{(\theta)}$, with $\theta\geq 0$.
We keep the notations of Section
\ref{sec:backbones}: we set $B_0(\tau)=\lb\partial,\partial^*\rb$
and set $(\tau_i, \, i\in I_0)$ the closures of the connected components
of $\tau\setminus B_0(\tau)$ viewed as trees in $\T$ rooted
respectively at point $x_i\in B_0(\tau)$ so that
$\tau=B_0(\tau)\circledast_{i\in I_0}(\tau_i,x_i).$

\begin{lem}\label{lem:ppp}
  Let  $\theta\geq   0$.  Under   $\N^{(\theta)}$,  the   point  measure
  $\sum_{i\in I_0}\delta_{(h-u_i-H(\tau_i),\tau_i)}$ on $[0,h]\times \T$
  is, conditionally given $\{H(\tau)=h\}$,  a Poisson point measure with
  intensity
\begin{equation}
   \label{eq:intensiteppp}
2 \beta \ind_{(0,h)}(u)\,  du\,\N^{(\theta)}[d\bt,\ H(\bt)\le h-u].
\end{equation}
\end{lem}

\begin{proof}
By the Williams decomposition, see \cite{ad:wdlcrtseppnm}, the point measure $\sum_{i\in
  I_0}\delta_{(u_i,\tau_i)}$ is under $\N^{(\theta)}$, conditionally given
$\{H(\tau)=h\}$, a Poisson point measure with intensity
\reff{eq:intensiteppp}. 
Then, for every non-negative function $\varphi$ on $[0,h]\times \T$, we
have
\begin{multline*}
\N^{(\theta)}  \left[\expp{-\sum_{i\in
      I_0}\varphi(h-u_i-H(\tau_i),\tau_i)}\Bigm| H(\tau)=h\right]\\
\begin{aligned}
& =\exp\left(-\int_0^h
  2\beta du\,\N^{(\theta)}\left[\left(1-\expp{-\varphi(h-u-H(\tau),\tau)}\right)\ind_{\{
    H(\tau)\le h-u\}}\right]\right)\\ 
&=\exp\left(-2\beta
  \N^{(\theta)}\left[\int_0^{h-H(\tau)}du \left(1-\expp{-\varphi(h-
        u-H(\tau),\tau)}\right)\ind_{\{H(\tau)\leq h\}}\right]\right)\\  
&=\exp\left(-2\beta \N^{(\theta)}\left[\int_0^{h-H(\tau)}dv
    \left(1-\expp{-\varphi(v,\tau)}\right)\ind_{\{H(\tau)\leq
      h\}}\right]\right)\\ 
& =\exp\left(-\int_0^h
  2 \beta dv\,\N^{(\theta)}\left[\left(1-\expp{-\varphi(v,\tau)}\right)\ind_{\{ H(\tau)\le
    h-v\}}\right]\right), 
\end{aligned}
\end{multline*}
where we performed the change of variables $v=h-u-H(\tau)$ for the third
equality. The lemma follows.
\end{proof}

\begin{theo}\label{theo:reversed_CRT}
Let $\theta\ge 0$. Let $\tau$ be a Brownian CRT under the excursion measure $\N^{(\theta)}$. Then, $\cR(\tau)$ is distributed as $\tau$.
\end{theo}

\begin{proof}
To prove the theorem, it suffices to prove, using Lemma
\ref{lem:B_n-->t},  that for every $n\in\N$,
$B_n(\tau)$ and $\cR_n(\tau)$ are equally distributed, which we prove
by induction.

First, as $\N^{(\theta)}$-a.e. $\tau\in\T_0$, we have
$B_0(\tau)=\cR_0(\tau)$ (viewed as equivalence classes). They
have consequently the same distribution.

Suppose now that $B_{n-1}(\tau)$ and $\cR_{n-1}(\tau)$ are equally
distributed for some $n\ge 1$. Recall that
\[
B_n(\tau)=B_0(\tau)\circledast_{i\in
  I_0}(B_{n-1}(\tau_i),x_i)\quad\mbox{and}\quad\cR_n(\tau)=\cR_0(\tau)\circledast_{i\in 
  I_0}(\cR_{n-1}(\tau_i),y'_i),
\]
where for  every $i\in I_0$, $y'_i$  is the unique point  of $B_0(\tau)$
which   has   the    same   height   as   $x_i^*$    i.e.    such   that
$h_\tau(y'_i)=h_\tau(x_i)+H(\tau_i)$.   Notice  that,  as  a  vertex  of
$\cR_0(\tau)$,              $y'_i$               has              height
$h_{\cR_0(\tau)} (y'_i)=H(\tau) - h_\tau(x_i) -H(\tau_i)$.

Thanks to  Lemma  \ref{lem:ppp}, conditionally  given $B_0(\tau)$,  the two families
$((h_\tau(x_i),\tau_i),         \,         i\in        I_0)$         and
$((h_{\cR_0(\tau)}  (y'_i),\tau_i),   \,  i\in   I_0)$  have   the  same
distribution.    By    the    induction   assumption,    the    families
$((h_\tau(x_i),B_{n-1}(\tau_i)), \, i\in              I_0)$              and
$((h_{\cR_0(\tau)}(y'_i),\cR_{n-1}(\tau_i)), \, i\in  I_0)$  have  also  the
same distribution. This implies that, under $\N^{(\theta)}$, $B_n(\tau)$
and $\cR_n(\tau)$ are equally distributed.
\end{proof}

The  reversal operation  is natural  on  the Brownian  CRT but  it has  no
elementary representation for the underlying Brownian excursion.

Recall the definition in Section \ref{sec:CRT} of the local time measure
$\ell_a(dx)$ of a Brownian CRT $\tau$ at level $a$. We denote by
$\ell_a(\tau)$ the total mass of this measure. We recover the
time-reversal distribution invariance of the local time of the Brownian
excursion. 

\begin{cor}\label{cor:reverse-local-time}
Let $\theta\geq 0$. $\N^{(\theta)}$-a.e., for every $a\ge 0$, $\ell_a(\tau)=\ell_{H(\tau)-a}(\cR(\tau))$.
\end{cor}

\begin{proof}
Let $a> 0$. Using Theorem 4.2 of \cite{dlg:pfalt}, we have that $\N^{(\theta)}$-a.e.:
\[
\ell_a(\tau  )=\lim_{\varepsilon\to 0}\frac{1}{\varepsilon}\Card\{x\in
r_\varepsilon(\tau),\ h_\tau(x)=a-\varepsilon\}
=\lim_{\varepsilon\to 0}\frac{1}{\varepsilon}\Card\{x\in
r_\varepsilon(\tau),\ h_\tau(x)=a\}.
\]
But,  by  construction,   we  have,  for  every   $\bt\in\T$  and  every
$\varepsilon>0$,
\[
\Card(\{x\in
r_\varepsilon(\bt),\ h_\bt(x)=a-\varepsilon\})
=
\Card(\{x\in
r_\varepsilon(\cR(\bt)),\ h_{\cR(\bt)}(x)=H(\bt)-a\}).
\]
Therefore, we have that for every $a> 0$, $\N^{(\theta)}$-a.e., $\ell_a(\tau)=\ell_{H(\tau)-a}(\cR(\tau))$.
Then, consider the continuous version of the local time to conclude.
\end{proof}

\subsection{Extension to a forest}\label{sec:extension}

\bigskip
For $\theta\geq 0$, we define the Brownian forest as the forest
$\cf=((h_i,\tau_i), \, i\in I)$ where $\sum_{i\in I} \delta_{h_i, \tau_i}$ is a Poisson point measure on
$\R\times \T$ with intensity $2\beta dh\, \N^{(\theta)}[d\tau]$ and we
denote by $\P^{(\theta)}$ its distribution.

\begin{rem}
   \label{rem:CD}
This
Brownian forest can be viewed as the genealogical tree of a stationary
continuous-state branching process (associated with the branching
mechanism $\psi_\theta$ defined in \reff{eq:psi}), see
\cite{cd:spsmrcatsbp}. To be more precise, for every $i\in I$ let 
$(\ell^{(i)}_a)_{a\ge 0}$ be the local time measures of the tree
$\tau_i$. For every $t\in\R$, we define the size $Z_t$ of the population
at time $t$ by 
\begin{equation}\label{eq:Zt}
Z_t=\sum_{i\in I}\ell^{(i)}_{t-h_i}(\tau_i),
\end{equation}
where we recall that the local time $\ell_a(\tau)$ of the CRT $\tau$ is
zero for $a\not\in [0, H(\tau)]$. For $\theta=0$, we have $Z_t=+\infty$ a.s. for every $t\in\R$. For $\theta>0$, the process $(Z_t, {t\ge 0})$ is a stationary Feller diffusion,
solution of the SDE 
\[
dZ_t=\sqrt{2\beta Z_t}\, dB_t+2\beta(1-\theta Z_t)dt.
\]
\end{rem}

A forest $\bff=  ((h_i,\bt_i), \, i\in I)$ is said  to be height-regular
if:
\begin{itemize}
\item for every $i\in I$, $\bt_i\in\T_0$;
\item for   every   $i,j\in   I$,  if  $i\ne   j$, then  $h_i\ne   h_j$   and
  $h_i+H(\bt_i)\ne h_j+H(\bt_j)$.
\end{itemize}

We define the reverse of a height-regular forest $\bff=((h_i,\bt_i), \, i\in I)$ as the forest
\[
\cR(\bff)=((-h_i-H(\bt_i),\cR(\bt_i)), \, i\in I).
\]

\begin{lem}\label{lem:reversed_forest}
Let $\theta\geq 0$. Let $((h_i,\tau_i), \, i\in I)$ be a Brownian forest under $\P^{(\theta)}$. Then the point process
\[
\sum_{i\in I}\delta _{(-h_i-H(\tau_i),\tau_i)}(dh, d\bt)
\]
is a Poisson point process on $\R\times \T$ with intensity $2\beta dh\,
\N^{(\theta)}[d\bt]$. 
\end{lem}

\begin{proof}
The proof is similar to the one of Lemma \ref{lem:ppp}. 
For every non-negative measurable function $\varphi$ on $\R\times \T$, we have, denoting $\E^{(\theta)}$ the expectation under $\P^{(\theta)}$,
\begin{align*}
\E^{(\theta)}\left[\expp{-\sum_{i\in
  I}\varphi(-h_i-H(\tau_i),\tau_i)}\right] 
& =\exp\left(-\int_{-\infty}^{+\infty}2\beta dh\, 
\N^{(\theta)}\left[1-\expp{-\varphi(-h-H(\tau),\tau)}\right]\right)\\
& =\exp\left(-2\beta \N^{(\theta)}
  \left[\int_{-\infty}^{+\infty}\left(1-\expp{-\varphi(-h-H(\tau),\tau)}\right)
dh\right]\right)\\ 
& =\exp\left(-2\beta
  \N^{(\theta)}\left[\int_{-\infty}^{+\infty} 
\left(1-\expp{-\varphi(v,\tau)}\right)dv\right]\right), 
\end{align*}
by an obvious change of variables, which yields the result.
\end{proof}

We deduce from Lemma \ref{lem:reversed_forest}, Lemma \ref{lem:reversed_CRT} and Theorem \ref{theo:reversed_CRT} the following corollary. 

\begin{cor}
   \label{cor:-forest}
   Let  $\theta\geq   0$.  Let   $\cf$  be   a  Brownian   forest  under
   $\P^{(\theta)}$. Then $\cf$  is a.s. height regular  and the reversed
   forest $\cR(\cf)$ is distributed as $\cf$.
\end{cor}

\begin{rem}
   \label{rem:bd}
For every real numbers $s<t$, and every forest $\cf$, we set
\[
M_s^t(\cf)=\Card(\{x\in\cf_\R,\ h_\cf(x)=s\mbox{ and }H(\bt_x)\ge t-s\})
\] 
the number of vertices of $\cf$ at height $s$ that have descendants at time $t$ (excluding the infinite spine).
Corollary \ref{cor:-forest} allows to straightforward recover (and
understand) Theorem 4.3  from \cite{bd:tl} that states that, under
$\P^{(\theta)}$ for any $\theta>0$, the processes
$(M_s^{s+r},s\in\R,r\ge 0)$ and $(M_{s-r}^s,s\in\R,r\ge 0)$ are equally
distributed. Indeed, to recover this result, it is enough to notice that a.s.:
\[
(M_s^{s+r}(\cf),s\in\R,r\ge 0)=(M_{s-r}^s(\cR(\cf)),s\in\R,r\ge 0).
\]
\end{rem}

\bibliographystyle{abbrv}
\bibliography{biblio}

\end{document}